\documentclass[12pt]{amsart}
\usepackage{amsthm,amssymb}
\usepackage{epsfig}
\usepackage[arrow]{xy}

\usepackage{amssymb,latexsym,cite,epsf,graphics}

\usepackage[margin=1.0in]{geometry}
\usepackage{graphicx,color}

\definecolor{darkred}{rgb}{1,0,0}

\newtheorem{theorem}{Theorem}[section]

\newtheorem{lemma}[theorem]{Lemma}

\newtheorem{proposition}[theorem]{Proposition}

\newtheorem{corollary}[theorem]{Corollary}

\newtheorem{conjecture}[theorem]{Conjecture}

\theoremstyle{remark}

\theoremstyle{definition}

\newtheorem{remark}[theorem]{Remark}
\newtheorem{example}[theorem]{Example}

\newtheorem*{theorem*}{Theorem}

\def\A{\hat{A}} 
\def\m{\mu_{\rightarrow}}

\def\CC{\mathbb{C}}

\def\SL{\operatorname{SL}}

\title{
Zamolodchikov integrability via rings of invariants
}

\numberwithin{equation}{section}

\begin{document}

\author{Pavlo Pylyavskyy}
\address{\hspace{-.3in} Department of Mathematics, University of Minnesota,
Minneapolis, MN 55414, USA}
\email{ppylyavs@umn.edu}

\date{\today
}

\thanks{Partially supported by NSF grants  DMS-1148634, DMS-1351590, and Sloan Fellowship.
}

%\subjclass{
%Primary
%13F60, % Cluster algebras
%Secondary
%05E99, % Algebraic combinatorics: None of the above, but in this
%13A50, % Actions of groups on commutative rings; invariant theory
%15A72. % Vector and tensor algebra, theory of invariants 
%}


\begin{abstract}
Zamolodchikov periodicity is periodicity of certain recursions associated with box products $X \square Y$ of two finite type Dynkin diagrams.  
%It was conjectured by Zamolodchikov in his study of thermodynamic Bethe ansatz. 
%In full generality the conjecture was proved by Keller using a categorification of cluster algebras. 
We suggest an affine analog of Zamolodchikov periodicity, which we call Zamolodchikov integrability. We conjecture that it holds for products $X \square Y$, where $X$ is a finite type Dynkin diagram and $Y$ is a bipartite extended Dynkin diagram.
We prove this conjecture for the case of $A_m \square A_{2n-1}^{(1)}$. The proof employs cluster structures in certain classical rings of invariants, previously studied by S.~Fomin and the author. 
\end{abstract}

\ \vspace{-.1in}

\maketitle

\tableofcontents

\section{Introduction}

\subsection{Cluster algebras}

A {\it {quiver}} $Q$ is a directed graph without loops or directed $2$-cycles. Some vertices of $Q$ are declared {\it {mutable}}, others are {\it {frozen}}. For a mutable vertex $z$ of quiver $Q$ one can define {\it {quiver mutation}} at $z$ as follows:
\begin{itemize}
 \item for each pair of edges $y \rightarrow z$ and $z \rightarrow y'$ create an edge $y \rightarrow y'$;
 \item reverse direction of all edges adjacent to $z$;
 \item if some directed $2$-cycle is present, remove both of its edges; repeat until there are no more directed $2$-cycles.
\end{itemize}

One can now define algebraic dynamics of {\it {seed mutations}} as follows. Each vertex $z$ of quiver $Q$ has an associated variable belonging to certain fixed ground field $\mathbb F$. By abuse of notation we also denote this variable $z$.
When $Q$ is mutated at $z$, the variable $z$ changes, with the new value $z'$
satisfying 
$$z z' = \prod_{y \rightarrow z} y + \prod_{z \rightarrow y} {y},$$
where the two products are taken over all edges directed towards $z$ and out of $z$, respectively. 
We denote $\mu_z$ the mutation at $z$, affecting both the quiver and the associated variable. 

  \begin{figure}[ht]
    \begin{center}
\vspace{-.1in}
%\scalebox{1}{
\epsfig{file=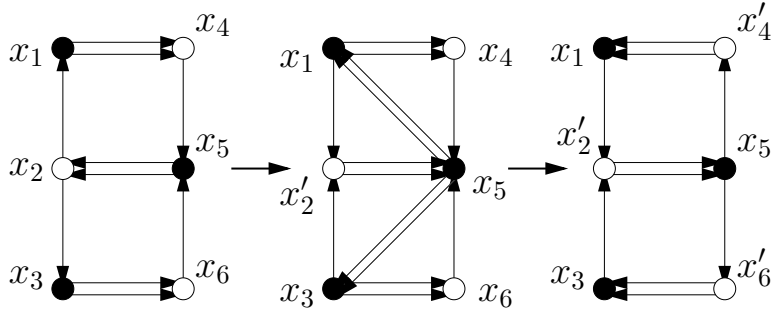, scale=1}
%\input{zint2.pstex_t} 
%}
\vspace{-.1in}
    \end{center} 
    \caption{A quiver mutation at $x_2$, followed by mutations at $x_4$ and $x_6$.}
    \label{fig:zint2}
\end{figure}

\begin{example}
 Consider quiver $Q$ on the left in Figure~\ref{fig:zint2}. 
When mutated at the middle vertex on the left, the new variable $x_2'$ is given by
$$x_2' = \frac{x_5^2 + x_1 x_3}{x_2}.$$
\end{example}

\subsection{Box product of quivers}

A quiver is a directed graph. Let $$Q = Q_0 \sqcup Q_1 \text{ and } Q' = Q_0' \sqcup Q_1'$$ be two quivers such that the underlying graph is bipartite. Assume that all edges in them are between $Q_0$ and $Q_1$, respectively, $Q_0'$ and $Q_1'$.  
Following the exposition in \cite{W}, define {\it {box product}} $Q \square Q'$ as follows.

\begin{itemize}
 \item The vertices are pairs $(q,q') \in Q \times Q'$.
 \item For each edge connecting $q'_1$ and $q'_2$ in $Q'$ and each $q \in Q$ an edge connects $(q,q'_1)$ and $(q,q'_2)$;  for each edge connecting $q_1$ and $q_2$ in $Q$ and each $q' \in Q'$ an edge connects $(q_1,q')$ and $(q_2,q')$.
 \item The directions are from $Q_0 \times Q'_0$ to $Q_1 \times Q'_0$, from $Q_1 \times Q'_0$ to $Q_1 \times Q'_1$, from $Q_1 \times Q'_1$ to $Q_0 \times Q'_1$, from $Q_0 \times Q'_1$ to $Q_0 \times Q'_0$.
\end{itemize}

 \begin{figure}[ht]
    \begin{center}
\vspace{-.1in}
%\scalebox{1}{
%\input{zint1.pstex_t} 
\epsfig{file=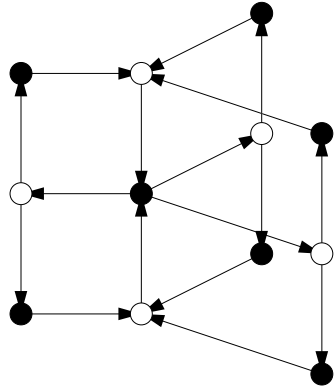, scale=1}
%}
\vspace{-.1in}
    \end{center} 
    \caption{An example of box product $D_4 \square A_3$.}
    \label{fig:zint1}
\end{figure}

An example of box product of two Dynkin diagrams can be seen in Figure~\ref{fig:zint1}.

\subsection{Zamolodchikov periodicity: $T$-system formulation}

In this section we give a brief overview of the $x$-variable, also known as $T$-system formulation of Zamolodchikov periodicity. For more detail and for more traditional $y$-variable formulation we refer the reader to an excellent exposition in \cite{W}.

The following lemma is easily verified from the definition of seed mutation.

\begin{lemma} \label{lem:comm}
 Mutations at two vertices of a quiver that are not connected by an edge commute.
\end{lemma}

Consider the bipartite coloring of $Q \square Q'$ where vertices in $Q_0 \times Q'_0$ and $Q_1 \times Q'_1$ are colored black, while vertices in $Q_1 \times Q'_0$ and $Q_0 \times Q'_1$ are colored white. Define $\mu_+$ and $\mu_-$to be the result of mutating $Q \square Q'$ 
at vertices of a particular color:
$$\mu_+ = \prod_{i \text{ black }} \mu_i, \;\;\; \mu_- = \prod_{i \text{ white }} \mu_i$$
Thanks to Lemma \ref{lem:comm} we do not need to specify the order of mutations in each case, since vertices of the same color are not connected by an edge. 
It is easy to check that the result of applying either $\mu_+$ or $\mu_-$ to $Q \square Q'$ is the same quiver but with directions of arrows reversed. 
An example can be seen in Figure~\ref{fig:zint2}.
The combined map $\mu_+ \mu_-$ then returns the original quiver $Q \square Q'$, including the arrow orientations.

\begin{theorem} \cite{K} \label{thm:zper}
 If $Q$ and $Q'$ are two finite type Dynkin diagrams, and $h$ and $h'$ are the corresponding Coxeter numbers, then $(\mu_- \mu_+)^{h+h'}$ is an identity transformation on the level of cluster variables. 
\end{theorem}

 \begin{figure}[ht]
    \begin{center}
\vspace{-.1in}
%\scalebox{1}{
%\epsfig{file=zint3.ps, scale=1}
\input{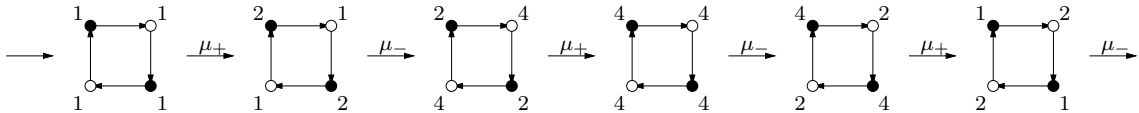} 
%}
\vspace{-.1in}
    \end{center} 
    \caption{An example of Zamolodchikov periodicity for $A_2 \square A_2$.}
    \label{fig:zint3}
\end{figure}

\begin{example}
 Figure~\ref{fig:zint3} shows an example of Zamolodchikov periodicity for type $A_2 \square A_2$. For brevity we started with the initial variables having value $1$, rather than writing formulas for general initial choice of variables. The Coxeter number for $A_2$ root system is $3$. In this case 
 the period of the system happens to be equal to $3$, i.e. half of the $h + h' = 3+3$ predicted by the theorem. 
\end{example}

\begin{remark}
Zamolodchikov periodicity was conjectured in \cite{Z} by Zamolodchikov for $Y$-systems of simply laced Dynkin diagrams. It was generalized by Ravanini-Valleriani-Tateo \cite{RVT}, Kuniba-Nakanishi \cite{KN}, Kuniba-Nakanishi-Suzuki \cite{KNS}, Fomin-Zelevinsky \cite{FZy}. 
Its special cases were proved by Frenkel-Szenes \cite{FS}, Gliozzi-Tateo \cite{GT}, Fomin-Zelevinsky \cite{FZy}, Volkov \cite{V}, Szenes \cite{S}. In full generality it was proved by Keller \cite{K} and later in a different way by  Inoue-Iyama-Keller-Kuniba-Nakanishi \cite{IIKKN1, IIKKN2}.
We refer the reader to \cite{W} for more details. 
 \end{remark}

\begin{remark}
This result is usually stated in terms of $y$-variable dynamics, see \cite{FZ1, FZ4} for definitions. We make the following remarks about the $x$-variable formulation. 
\begin{itemize}
 \item The $x$-variable formulation implies the $y$-variable formulation. This is seen using the explicit formulas for $y$-dynamics derived by Fomin and Zelevinsky in terms of F-polynomials, see \cite[Proposition 3.9]{FZ4}.
 \item In type $A_n \square A_m$ an elegant proof can be given using cluster structure in Grassmannians. This proof in essence lifts Volkov's argument for $y$-variable case \cite{V} to the level of $x$ variables. The details of the proof will appear in \cite{FWZ}, 
 cf. \cite{FMSRI}. At the moment, an argument which is very close in nature can be found in \cite{GR}.
 \item A different proof for $A_n \square A_m$ types can be found in \cite{DK}.  
\end{itemize}
\end{remark}

\subsection{Zamolodchikov integrability}

 Call a sequence $a_{n}, n \in \mathbb Z$ {\it {linearizable of order $k$}} if it satisfies a linear recurrence relation
 $$a_{n} = c_1 a_{n-1} + \ldots + c_{k} a_{n-k},$$
 for some fixed choice of $k$ and constant coefficients $c_i$.
Let  $M_{\ell, K} = \{m_{ij}\}$ be a $K \times K$ Toeplitz matrix with $m_{ij} = a_{\ell+i-j}$. 
The following lemma is easy to verify.

\begin{lemma}
If $a_{n}$ is linearizable, then for any $K > k$ we have $\det(M_{\ell, K})=0$ for any $\ell \in \mathbb Z$. 
\end{lemma}

\begin{example}
 The sequence $\{F_i\}_{i \geq 1}$ of Fibonacci numbers $F_{i} = F_{i-1} + F_{i-2}$, $F_1 = F_2 = 1$ is linearizable, and it satisfies 
$$ \det
\left(
\begin{matrix}
F_{i+2} &  F_{i+3}  & F_{i+4}\\
F_{i+1}  &  F_{i+2} & F_{i+3}\\
F_i  &   F_{i+1}  & F_{i+2}
\end{matrix} \right) = 0
$$
 for any $i$.  Fibonacci numbers are linearizable of order $2$, but (as it is easy to check) not of order $1$. 
\end{example}

Now, assume we have a field automorphism $\mu$ acting on the field $\mathbb F$. 
%generated by a distinguished collection of variables $x_i$. 
We say that $\mu$ is {\it {linearizable}} at $x \in \mathbb F$ if  the sequence $\ldots, \mu^{-1}(x), x, \mu(x), \mu^2(x), \ldots$ is linearizable. 
If $\mathbb F$ comes with a distinguished choice of generators, for example variables of the initial cluster of some cluster algebra, we say that $\mu$ is 
{\it {linearizable}} if it is linearizable at each of the distinguished generators. In cases when $\mu = \mu_- \mu_+$ for some quiver, we shall say that the quiver is 
{\it {Zamolodchikov integrable}}.

\begin{remark}
 The idea of linearizability as integrability is certainly not new, even in the cluster context - see for example \cite{DK, FH}. One can aim at this property in a very general setting of Nakanishi's generalized $T$-systems \cite{N}.
 We introduce term {\it {Zamolodchikov integrability}} for the following reason. In the most general setting of Nakanishi's $T$-systems it seems hard to even conjecture when linearizability holds. On the other hand, 
 when one restricts oneself to bipartite dynamics a la Zamolodchikov, precise classification results can be stated, see Conjecture \ref{conj:lim}. The name of Zamolodchikov seems to be the most natural choice to capture this nice special case with a single term. 
\end{remark}

\begin{conjecture} \label{conj:main}
  If $X$ is a finite type Dynkin quiver and $Y$ is an affine type extended Dynkin quiver (with bipartite underlying graph), then the map $\mu_- \mu_+$ defined above is linearizable 
  when acting on the field of rational functions in the initial cluster variables of quiver $X \square Y$.  
\end{conjecture}

\begin{example} \label{ex:1}
 As one starts applying operators $\mu = \mu_- \mu_+$ to the quiver $A_3 \square A_1^{(1)}$ in Figure~\ref{fig:zint2}, one obtains at vertices $x_1$, $x_3$, $x_4$ and $x_6$ sequences that are linearizable of order $4$. 
 For example, if one starts with all variables equal to $1$, then $$x_1, \mu(x_1), \mu^2(x_1), \ldots = 1, 2, 22, 377, 7110, 136513, 2629418, 50674318, 976694489, \ldots$$ and 
 $$ \det
\left(
\begin{matrix}
7110 &  136513  & 2629418 &  50674318  & 976694489 \\
377 &  7110  & 136513 &  2629418  & 50674318 \\
22 &  377  & 7110 &  136513  & 2629418 \\
2 &  22  & 377 &  7110  & 136513 \\
1 &  2  & 22 &  377  & 7110 
\end{matrix} \right) = 0
$$
\end{example}

Note that different nodes of the quiver may be linearizable of different orders, as one can check nodes $x_2$ and $x_5$ are {\it {not}} linearizable of order $4$ in Example \ref{ex:1}. 
This differs from Zamolodchikov periodicity, where periods of all nodes coincide. 

Now we are ready to state our main theorem.

\begin{theorem} \label{thm:main}
Conjecture \ref{conj:main} holds in the case of quivers $A_m \square A_{2n-1}^{(1)}$. 
\end{theorem}

In fact, we can give a bound on the order of linearizability. Let $q$ be the $j$-th vertex in $A_m$ counting from one of the ends, and let $q'$ be any vertex of $A_{2n-1}^{(1)}$. 
%Assume that $j = \delta n + \eta$, where $\delta, \eta \in \mathbb Z_{\geq 0}$ and $0 \leq \eta < n$. 

\begin{theorem} \label{thm:spd}
In the case of quiver $A_m \square A_{2n-1}^{(1)}$, vertex $(q,q')$  is linearizable of order $$n \cdot {m+1 \choose j}.$$  
\end{theorem}

\begin{example}
 In Example \ref{ex:1} we saw that for $m=3, n=1$ and $j=1$ the sequence is linearizable of order $4 = 1 \cdot {3+1 \choose 1}$. For $j=2$ one can check that $x_2$ and $x_5$ are linearizable of order 
 $6 = 1 \cdot {3+1 \choose 2}$ predicted by the theorem. 
\end{example}

\begin{remark}
 Zamolodchikov integrability for the cases $A_1 \square Y$ for affine extended Dynkin diagrams $Y$ was proven in \cite{ARS2} for types $A$ and $D$, and in \cite{KS} in full generality. Thus, our work can be considered a generalization of their work. 
\end{remark}

\begin{example}
 Assuming $m=1$ in the formula above forces $j=1$. The resulting order $2n$ then agrees with \cite{ARS2, KS}.
\end{example}

\begin{remark}
In the case $n=1$ one obtains the {\it {$Q$-systems}}, Theorem \ref{thm:spd} in this case yields formula $m+1 \choose j$. $Z$-integrability in this situation in the special case $j=1$ was proven by DiFrancesco and Kedem in \cite{DK2}. 
Thus, our work can be considered a generalization of their work. 

The same authors also consider the dynamics of $A_m \square A_{2n-1}^{(1)}$ and $A_{2m-1}^{(1)} \square A_{2n-1}^{(1)}$ quivers in \cite{DK}, obtaining explicit formulas for variables. 
\end{remark}

\begin{remark}
 Zamolodchikov integrability does {\it {not}} work for box products $X \square Y$ of affine extended Dynkin diagrams. Indeed, consider the simplest case of $A_{1}^{(1)} \square A_{1}^{(1)}$. If one starts with all initial variables equal to 
 $1$, one obtains the following sequence of values at one of the vertices: $$x_1, \mu(x_1), \mu^2(x_1), \ldots = 1,2,64, \ldots,$$ with the formula $\mu^k(x_1) = 2^{(k+1)(2k+1)}$ for a general term with $k \geq 1$. This expression grows super exponentially, and thus 
 cannot be a solution to a linear recurrence relation. 
 
 It seems likely that the kind of integrability that works in this case is the Arnold-Liouville integrability, cf. \cite{GSTV, GK}. In particular, quivers $A_{2m-1}^{(1)} \square A_{2n-1}^{(1)}$ fit naturally on a torus, and thus methods of either Gekhtman, Shapiro, Tabachnikov and Vainshtein \cite{GSTV} 
 or of Goncharov and Kenyon \cite{GK} should allow one to create an appropriate Poisson bracket, etc. We do not pursue this direction in this paper. 
\end{remark}

\subsection{Beyond box products}

We are going to state a conjectural criterion for when Zamolodchikov periodicity or integrability phenomena occur for more general quivers. 
Let $Q$ be a quiver such that the underlying graph is bipartite, and such that mutating at its black vertices, followed by mutating at its white vertices returns the same quiver. We are still going to denote the combination of those two operations as $\mu_- \mu_+$, 
and we call such quivers {\it {recurrent}}.

Let us say that labelling $\nu \colon Q \rightarrow \mathbb R_{>0}$ of vertices of $Q$ with positive real numbers is {\it {subadditive}} if the following conditions hold:
\begin{itemize}
 \item for any vertex $z$ we have $$\nu(z) \geq \frac{1}{2} \max \left(\sum_{y \rightarrow z} \nu(y), \sum_{z \rightarrow y} \nu(y) \right),$$ where the sums are taken over all incoming, resp. outgoing arrows of $z$;
 \item if the equality $\nu(z) = \frac{1}{2} \max(\sum_{y \rightarrow z} \nu(y), \sum_{z \rightarrow y} \nu(y))$ holds, it must be the case that $$\sum_{y \rightarrow z} \nu(y) \not =  \sum_{z \rightarrow y} \nu(y).$$
\end{itemize}
Let us say that a labelling is {\it {strictly subadditive}} if the inequality  of the first condition is always strict (and thus the second condition never applies). 
Let us say that a labelling is {\it {weakly subadditive}} if the first condition (with weak inequality) holds, regardless of whether the second condition holds or not. 

\begin{conjecture} \label{conj:lim}
 A recurrent quiver is Zamolodchikov integrable if and only if a subadditive labelling exists. 
\end{conjecture}

\begin{conjecture} \label{conj:sup}
 A recurrent quiver is Zamolodchikov periodic if and only if a strictly subadditive labelling exists. 
\end{conjecture}

\begin{conjecture} \label{conj:tol}
 A recurrent quiver is Arnold-Liouville integrable if and only if a weakly subadditive labelling exists. 
\end{conjecture}

 \begin{figure}[ht]
    \begin{center}
\vspace{-.1in}
%\scalebox{1}{
\epsfig{file=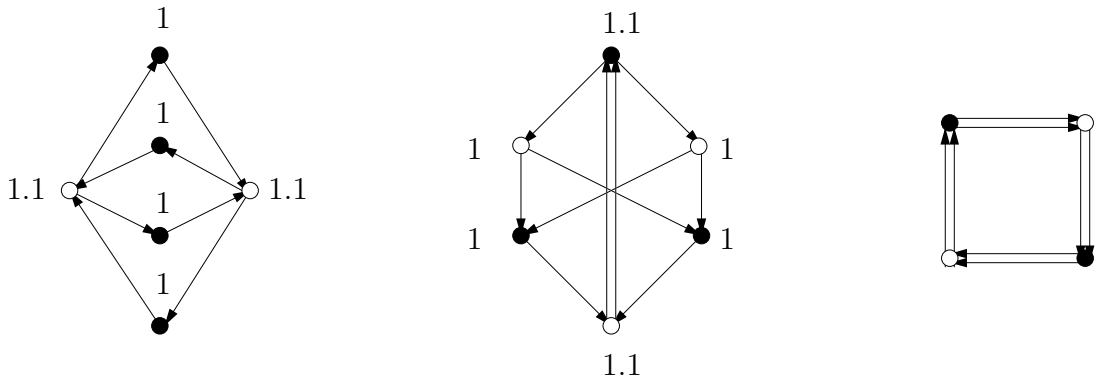, scale=1}
%\input{zint4.pstex_t} 
%}
\vspace{-.1in}
    \end{center} 
    \caption{Some recurrent quivers.}
    \label{fig:zint4}
\end{figure}

\begin{example}
 The first quiver in Figure~\ref{fig:zint4} has a strictly subadditive labelling as shown. One can check that it is Zamolodchikov periodic, with $(\mu_- \mu_+)^4 = id$.
\end{example}

\begin{example}
 The second quiver in Figure~\ref{fig:zint4} has a subadditive labelling as shown. It does not have a strictly subadditive labelling however since the double edge forces its endpoints to have equal labellings. 
 One can check that it is Zamolodchikov integrable: four of its vertices are linearizable of order $8$, while the other two are linearizable of order $5$.
\end{example}

\begin{example}
 The third quiver $A_1^{(1)} \square A_1^{(1)}$ in Figure~\ref{fig:zint4} does not have a subadditive labelling. Indeed, the first condition would force all four vertices to have the same label, which would then violate the second condition. It has a  weakly 
 subadditive labelling however. This quiver occurs as a special case of construction in \cite{GSTV}, with all the associated Arnold-Liouville integrability implications.
\end{example}

In fact, the following proposition is not hard to verify, using Vinberg's \cite{Vi} {\it {additive}} and {\it {subadditive}} functions on Dynkin diagrams, see also \cite[Chapter 7.3]{ARS}.

\begin{proposition} The quiver $X \square Y$ has 
 \begin{itemize}
  \item a strictly subadditive labelling for $X$ and $Y$ finite type Dynkin diagrams;
  \item a subadditive labelling for $X$ a Dynkin diagram of finite type and $Y$ an extended Dynkin diagram of affine type;
  \item a weakly subadditive labelling exists if both $X$ and $Y$ are affine type extended Dynkin diagrams.
 \end{itemize}
\end{proposition}

\begin{remark}
 Reutenauer in \cite{R} uses additive and subadditive labellings for a similar purpose as we do here: to classify when growth of variables is exponential, and to deduce linearizability whenever this happens. He obtains a beautiful (essentially, Cartan-Killing) classification of the 
 cases when this happens. 
 
 The quiver dynamics he deals with however is such that the vertices at which one mutates are required to be either a source or a sink. This is closely related to the notion of bipartite quiver in the sense of Fomin and Zelevinsky \cite{FZ4}. 
 Thus, Reutenauer's classification does not address the dynamics we consider in this paper. 
\end{remark}

\begin{remark}
 The intuition behind the labelling conjectures is  $$\text{variable  } \sim \exp(\text{label } \times \text{ time}).$$ One may expect the same heuristic to work not only for the bipartite dynamics of repeating $\mu_- \mu_+$, but also for other sequences of mutations that at the end return the original quiver. 
 A special case of such mutations and their integrability properties were studied by Fordy, Marsh and Hone in \cite{FM, FH}. It would be interesting to see if some form of labelling criterion agrees with their classification of integrable cases. 
\end{remark}

\smallskip
\centerline{------------}
\smallskip

The author is grateful to Sergey Fomin,  Bernhard Keller, Vic Reiner, Christophe Reutenauer, Gwendolen Powell, Michael Gekhtman, and Pavel Galashin for their comments on the draft of the paper. The author would also like to express gratitude to the anonymous referees for their careful 
reading of the draft of the paper and many useful comments.

\section{Invariants and tensors} 

\subsection{Rings of $SL_{m+1}$ invariants}

Let $V\cong \mathbb C^{m+1}$ be a vector space endowed with a volume form.
The special linear group $\SL(V)$ acts 
on both $V$ and the dual space $V^*$, acting on the latter via 
$$(gu^*)(v)=u^*(g^{-1}(v)),$$
for $v\in V$, $u^*\in V^*$, and $g\in\SL(V)$. 
The group $\SL(V)$ also acts on itself, 
via conjugation. 
Following \cite{FP2}, we define the rings
\[
R_{a,b,c}(V)=\CC[(V^*)^a\times V^b \times (\SL(V))^c]^{\SL(V)}
\]
of $\SL(V)$-invariant polynomials on $(V^*)^a\times V^b \times (\SL(V))^c$. 
The closely related $\SL(V)$ action on the ring
$\CC[(V^*)^a\times V^b \times\operatorname{End}(V)]$ was studied by
  Procesi~\cite{procesi}. 

\begin{theorem}[{\rm cf.\ \cite[Theorem 12.1]{procesi}}]
\label{thm:procesi}
 The ring of invariants $R_{a,b,c}(V)$ is generated by: 
\begin{itemize}
 \item the traces $tr(X_{i_1} \dotsc X_{i_r})$ of arbitrary
 (non-commutative) monomials in the $c$ matrices in $\SL(V)$;
 \item the pairings $\langle v_i, M w_j \rangle$, where $v_i$ is a
   vector, $w_j$ is a covector and $M$ is any monomial as before;
 \item volume forms $\langle M_1 v_{i_1}, \ldots, M_n v_{i_n}\rangle$,
 where $M_i$-s are monomials as before and $v_i$-s are vectors; 
 \item volume forms $\langle M_1 w_{i_1}, \ldots, M_n w_{i_n}\rangle$,
 where $M_i$-s are monomials as before and $w_i$-s are covectors. 
\end{itemize}
\end{theorem}

Of crucial importance in what follows will be the rings $R_{0,2n,1}(V)$.

\subsection{Tensor diagrams on surfaces}  \label{sec:ann}

We refer the reader to \cite{FP2} for more details of the following construction.

Let $S$ be a connected oriented surface with nonempty
boundary~$\partial S$ and 
finitely many marked points on~$\partial S$,
each of them colored black or white. 

Let us draw several simple non-intersecting
curves on~$S$ called {\it {cuts}} such that:
\begin{itemize}
 \item $S$ minus the cuts is homeomorphic to a disk;
 \item each cut connects unmarked boundary points; 
 \item for each cut, a choice of direction is made;
 \item each cut is defined up to isotopy that fixes its endpoints.
\end{itemize}

Let $a$ be the number of white boundary vertices, $b$ be the number of black boundary vertices, $c$ be the number cuts. 
We associate a covector in $(\mathbb C^{m+1})^*$ to each white point, a vector in $\mathbb C^{m+1}$ to each black point, and 
an element of $SL_{m+1}$ to each cut. 

A \emph{tensor diagram}
is a finite bipartite graph $D$ embedded in~$S$, with a fixed 
proper coloring of its vertices into two colors, black and white, 
such that each internal vertex is $(m+1)$-valent,
and each boundary vertex is a marked point of~$S$.
The embedded edges of~$D$ are allowed to cross each other.

We denote by $\operatorname{bd}(D)$ 
(resp. $\operatorname{int}(D)$) the set 
of \emph{boundary} (resp. \emph{internal}) vertices of~$D$. 

A tensor diagram $D$ on $S$
defines an $\SL(V)$ invariant
$[D]\in R_{a,b,c}$ via the following formula.
Let $\operatorname{cut}(D)$ denote the set
of points where $D$ crosses the cuts. 
\emph{Edge fragments} are the pieces into which those cuts cut the
edges of~$D$. 
(If an edge is not cut, it forms an edge fragment by itself.) 
Then the invariant $[D]$ is given by
\begin{equation*}
%\label{eq:[D]-in-coordinates}
\begin{split}
[D]=&\sum_\ell %{\ell\,:\,\operatorname{edge}(D)\to\{1,2,3\}}
\biggl(\,\prod_{v\in\operatorname{int}(D)}\operatorname{sign}(\ell(v))\biggr)
\biggl(\,\prod_{\substack{v\in\operatorname{bd}(D)\\
\text{$v$ black}}}x(v)^{\ell(v)}\biggr)\\ &\quad
\biggl(\,\prod_{\substack{v\in\operatorname{bd}(D)\\
\text{$v$ white}}}y(v)^{\ell(v)}\biggr)
\biggl(\,\prod_{\substack{v\in \operatorname{cut}(D)\\
}}X_{\ell(v)}\biggr)
\end{split}
\end{equation*}
where 
\begin{itemize}
\vspace{-.05in}
\item
$\ell$ runs over all labellings of the edge fragments in~$D$ by the
numbers~$1,\ldots, m+1$ such that for each internal vertex~$v$ of~$D$,
the labels of the edges incident to $v$ are 
distinct; 
\item
$\operatorname{sign}(\ell(v))$ is the sign of the 
permutation defined by the clockwise
% SF: or is it counterclockwise?
reading of those $m+1$ labels; 
%determined by the cyclic ordering of the edges incident to~$v$; 
\item
$x(v)^{\ell(v)}$ denotes the monomial $\prod_e x_{\ell(e)}(v)$,
product over all edges $e$ incident to~$v$, and similarly for $y(v)^{\ell(v)}$;
\item 
$X_{\ell(v)}$ is the entry $X_{ij}$ of the matrix 
$X \in \SL(V)$ associated with the crossing of 
the cut at a vertex~$v$; % $v\in \operatorname{cut}(D)$,
here $i$ and $j$ are the labels of the edge fragments adjacent to~$v$.
(Depending on the directions at the crossing, we may need to
invert~$X$.) 
\end{itemize}

 \begin{figure}[ht]
    \begin{center}
\vspace{-.1in}
%\scalebox{1}{
\epsfig{file=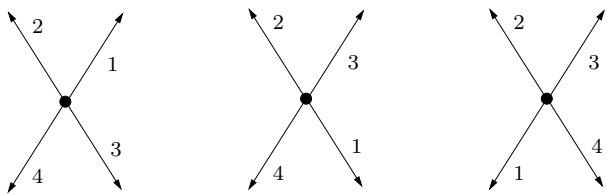, scale=1}
%\input{zint5.pstex_t} 
%}
\vspace{-.1in}
    \end{center} 
    \caption{Three edge fragment labellings around a vertex.}
    \label{fig:zint5}
\end{figure}

Note that for even $m$ the sign $\operatorname{sign}(\ell(v))$ is well defined as it is the same no matter where we start reading our permutation of edge fragment labels in cyclic order. For odd $m$ this is not the case however. We deal with it by assigning positive sign 
to a fixed base choice of edge fragment labelling. Then any other labelling has a well-defined sign at each vertex which is the product of the base choice of sign and the actual sign in any cyclic reading. For example, if $m=3$ and edges around an internal vertex are labelled with 
$1,2,3,4$ in the base labelling as shown in Figure~\ref{fig:zint5} on the left, then the labelling in the middle gets negative sign, while the labelling on the right gets positive sign. 

As a result, we only define $[D]$ up to a sign, unless we also specify the base labelling. In the actual cases we will deal with there will be a natural choice of base labelling, which will be indicated. 

\begin{example}
 The following figure shows an invariant $[D]$ in $R_{0,4,1}$ represented as a tensor diagram on an annulus. Two of the four vectors are placed on one boundary component and two on the other. 
 The cut represents an element $A \in SL_4$. 
  \begin{figure}[ht]
    \begin{center}
\vspace{-.1in}
%\scalebox{1}{
\epsfig{file=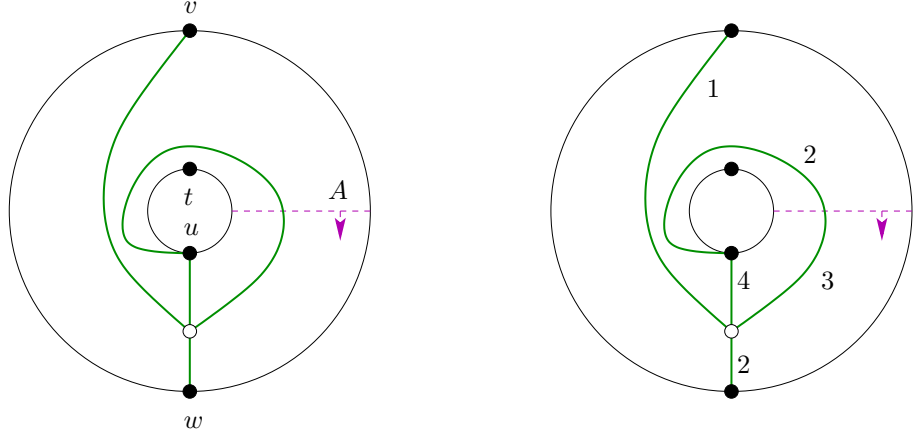, scale=1}
%\input{zint6.pstex_t} 
%}
\vspace{-.1in}
    \end{center} 
    \caption{A tensor diagram representing an invariant in $R_{0,4,1}$ for $m=3$, and one choice of labelling of the resulting five edge fragments.}
    \label{fig:zint6}
\end{figure}
The right side shows one possible labelling of edge fragments, resulting in contribution to $[D]$ equal to 
$v_1 u_2 u_4 w_2 a_{2,3}.$ The sign of the contribution is not determined since we did not specify the base labelling. Let us choose this particular labelling as the base one. Summing over all contributions one gets 
$$[D] = 
- \det
\left(
\begin{matrix}
v_1 &  w_1  & u_1 & (Au)_1\\
v_2 &  w_2  & u_2 & (Au)_2\\
v_3 &  w_3  & u_3 & (Au)_3\\
v_4 &  w_4  & u_4 & (Au)_4
\end{matrix} \right).
$$
\end{example}

\subsection{Normalization and skein relations}

We shall also consider {\it {normalized}} tensors associated with tensor diagrams as follows. Let 
$$[[D]] = [D] \prod \frac{1}{k!},$$
where the product is taken over all homotopy equivalence classes of edges connecting pairs of internal vertices in $D$,
%all pairs of internal vertices connected by an edge in $D$, 
and $k$ is the number of edges in such an equivalence class.
For example, for $D$ in Figure~\ref{fig:zint6} we have $[[D]]=[D]$, since even though there is a pair of vertices connected by two edges, those two vertices are not both internal, and even if they were - those two edges are not homotopy equivalent. 
  \begin{figure}[ht]
    \begin{center}
\vspace{-.1in}
%\scalebox{1}{
\epsfig{file=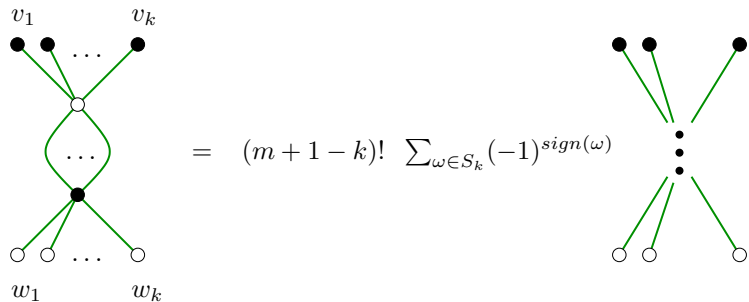, scale=1}
%\input{zint7.pstex_t} 
%}
\vspace{-.1in}
    \end{center} 
    \caption{A skein relation that can be applied locally.}
    \label{fig:zint7}
\end{figure}
On the other hand, for the tensor diagram $D$ on the left in Figure~\ref{fig:zint7} we have $$[[D]] = \frac{1}{(m+1-k)!} [D].$$
The right hand side of Figure~\ref{fig:zint7} shows how to express $[D]$ as an alternating sum over all possible ways to match the $k$ vertices on top with the $k$ vertices on the bottom. 
Note that this relation can be applied locally, i.e. vertices $v_i$ and $w_j$ may be internal as well as boundary.

\begin{remark}
An important special case of the relation is as follows: if two of (say) $v_i$-s coincide and are a boundary vertex, the resulting tensor vanishes, as evident from alternating nature of the skein relation. 
\end{remark}

\begin{remark}
 For odd $m$ one needs base labellings of the tensor diagrams to agree with each other in order for the signs on the right hand side of the relation to be as shown. In absence of specified base labellings, the skein relation can be considered to 
 hold up to a correct sign choice for each term. 
\end{remark}

For future use, let $$[[u_1, \ldots, u_{i}],u^*,[u_{i+1}, \ldots, u_{m+2}]] = [[D]],$$ where $D$ is the tensor diagram in Figure~\ref{fig:zint16}.   

  \begin{figure}[ht]
    \begin{center}
\vspace{-.1in}
%\scalebox{1}{
%\epsfig{file=zint16.ps, scale=1}
\input{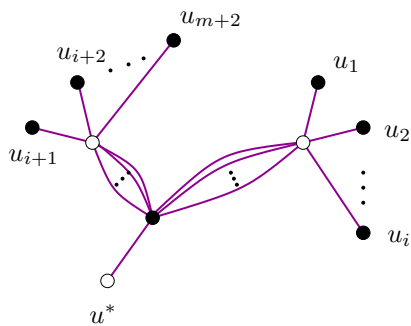} 
%}
\vspace{-.1in}
    \end{center} 
    \caption{An example of a tensor diagram.}
    \label{fig:zint16}
\end{figure}

\section{Proof of the main theorem}

\subsection{The initial cluster of type $A_m \square A_{2n-1}^{(1)}$}
Consider an annulus with $n$ black marked points placed on each of the two boundary components, $2n$ points total. 

Each marked point has a vector in $\mathbb C^{m+1}$ associated with it, denoted $v_1$ through $v_{n}$ on one boundary component, $w_1$ through $w_{n}$ on the other. The direction of numbering is counterclockwise on both components. 
In addition, we consider a cut associated with an element $A \in SL_{m+1}$ between the two 
boundary components. We assume that the cut separates $v_n$ with $v_1$ and $w_n$ with $w_1$. 
We set $v_i = A v_{i-n}$ and $w_i = A w_{i-n}$, thus extending indexing set of $v$-s and $w$-s to $\mathbb Z$.

  \begin{figure}[ht]
    \begin{center}
\vspace{-.1in}
%\scalebox{1}{
%\epsfig{file=zint9.ps, scale=1}
\input{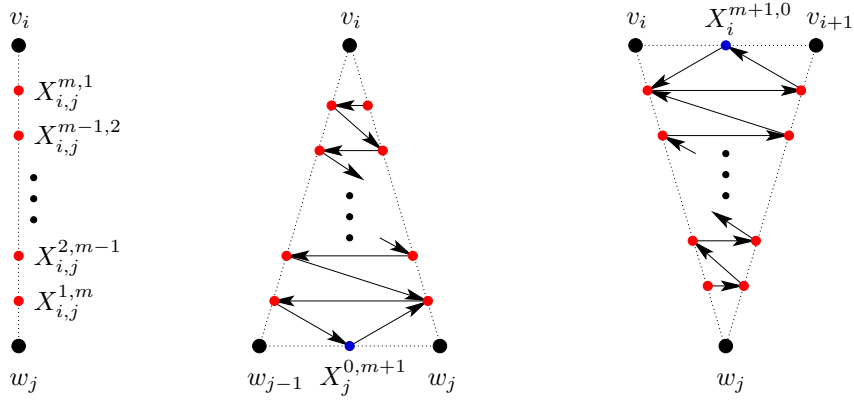} 
%}
\vspace{-.1in}
    \end{center} 
    \caption{Variables planted on segments of a triangulation; two types of narrow triangles and quivers built inside them.}
    \label{fig:zint9}
\end{figure}
We consider triangulations of the annulus by $2n$ segments of the form $v_i w_j$ into 
{\it {narrow triangles}}, i.e. triangles where two vertices on the same boundary component have adjacent indeces modulo $n$. To each such {\it {narrow triangulation}} $T$ we can associate a seed $\mathfrak T$ as follows. 
\begin{itemize}
 \item For each $i$ plant a frozen variable $$X_i^{m+1,0} = \langle v_i, v_{i+1}, \ldots, v_{i+m} \rangle$$ on the arc connecting $v_i$ and $v_{i+1}$.
 \item For each $i$ plant a frozen variable $$X_i^{0, m+1} = \langle w_i, w_{i-1}, \ldots, w_{i-m} \rangle$$ on the arc connecting $w_{i-1}$ and $w_{i}$.
 \item For each segment $v_i w_j$ of the triangulation plant variables $$X_{i,j}^{\alpha, \beta} = \langle v_i, v_{i+1}, \ldots, v_{i+\alpha-1}, w_j, w_{j-1}, \ldots, w_{j-\beta+1} \rangle $$ on this segment, as shown in Figure~\ref{fig:zint9}. Here we always have 
 $\alpha + \beta = m+1$.
\end{itemize}
In addition, in each narrow triangle create a quiver connecting planted functions as shown in Figure~\ref{fig:zint9}.  This creates a quiver on all of the variables, which is the final part of the seed $\mathfrak T$.

 The matrix $A$ is implicitly present in the definition since we use $v_i = A v_{i-n}$ and $w_i = A w_{i-n}$ throughout. In particular, if the diagonal $v_i w_j$ crosses the $A$ cut, we require the indexing to satisfy $$\left[\frac{i}{n}\right] - \left[\frac{j}{n}\right] = \text{number of 
 times the diagonal crosses the $A$-cut in positive direction},$$ where the number on the right is computed as we walk from $v$ end to $w$ end of the diagonal. The number can be negative if the crossing is in the negative direction. 
 
   \begin{figure}[ht]
    \begin{center}
\vspace{-.1in}
%\scalebox{1}{
\epsfig{file=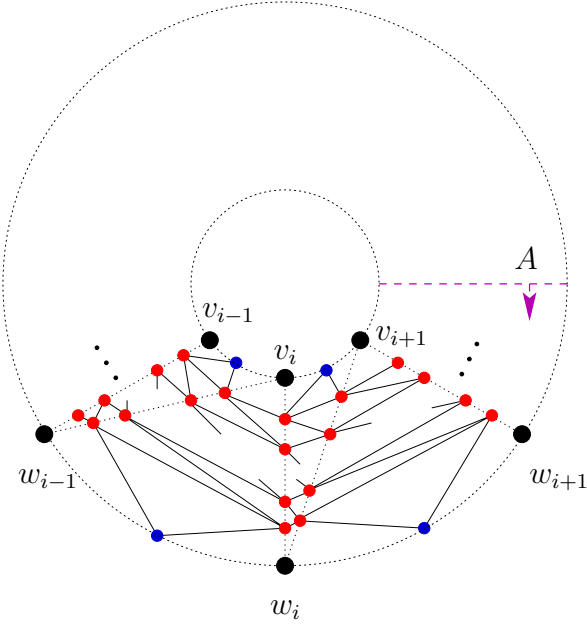, scale=1}
%\input{zint8.pstex_t} 
%}
\vspace{-.1in}
    \end{center} 
    \caption{The initial cluster of type $A_m \square A_{2n-1}^{(1)}$.}
    \label{fig:zint8}
\end{figure}
Now we can create the initial seed $\mathfrak T^*$ by taking the triangulation by all segments of the form $v_i w_i$ and $v_i w_{i-1}$. The result is shown in Figure~\ref{fig:zint8}.
\begin{lemma}
 The resulting quiver of $\mathfrak T^*$, ignoring the frozen variables, is of type $A_m \square A_{2n-1}^{(1)}$.
\end{lemma}
The proof of the lemma is clear from the construction. Note that there is more than one way to express the same variable. Specifically, $$X_{i,j}^{\alpha, \beta} = X_{i+n,j+n}^{\alpha, \beta}, \text{ since }  \langle A u_1, \ldots, A u_{m+1} \rangle = \langle u_1, \ldots, u_{m+1} \rangle.$$
\begin{example}
 For $n=m=2$ we obtain the following quiver. 
   \begin{figure}[ht]
    \begin{center}
\vspace{-.1in}
%\scalebox{1}{
\epsfig{file=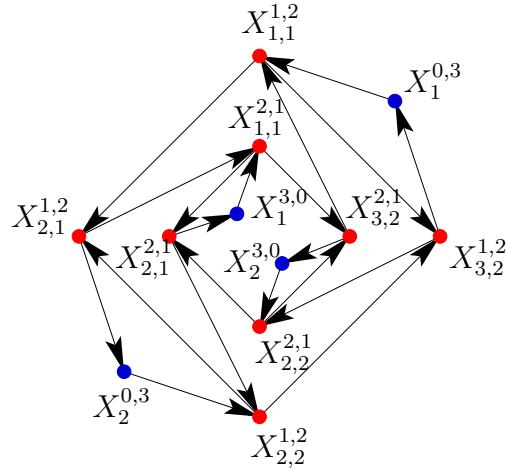, scale=1}
%\input{zint10.pstex_t} 
%}
\vspace{-.1in}
    \end{center} 
    \caption{An example of initial quiver for $n=m=2$.}
    \label{fig:zint10}
\end{figure}
For example, the diagonal connecting $v_1$ to $w_2$ crosses the cut once in positive direction. Thats why we choose to index it $v_3 w_2$, so that we have $[\frac{3}{2}] - [\frac{2}{2}] = 1$. Therefore, the two variables planted on this diagonal 
are $X_{3,2}^{2,1} = \langle A v_1, A v_2, w_2 \rangle$ and $X_{3,2}^{1,2} = \langle A v_1, w_2, w_1 \rangle$. 
\end{example}

Thus, we have realized the desired initial cluster inside the ring $R_{0,2n,1}$. In order to use this realization, the following property of $\mathfrak T^*$ is needed. Its proof is postponed until Section \ref{sec:aind}.

\begin{theorem} \label{thm:aind}
 Variables of the seed $\mathfrak T^*$ are algebraically independent. 
\end{theorem}

\subsection{Zamolodchikov $\mu_- \mu_+$ dynamics as triangulation evolution}
Now we argue that as we keep applying the mutation sequences $\mu_+$ and $\mu_-$, we keep getting seeds associated with narrow triangulations. This follows from the following lemma.

Assume a narrow triangulation $T$ contains diagonals $v_i w_j$, $v_{i+1} w_j$ and $v_{i+1} w_{j+1}$. 
\begin{lemma} \label{lem:locmut}
 Mutating all variables $X_{i+1,j}^{\alpha, \beta}$ in $\mathfrak T$ results in the variables $X_{i,j+1}^{\alpha, \beta}$ sitting on the diagonal $v_i w_{j+1}$ of the narrow triangulation obtained from $T$ by changing one diagonal. 
\end{lemma}

\begin{proof}
Let us identify $X_i^{m+1,0}$ with $X_{i,j}^{m+1,0}$ for any $j$, and similarly $X_j^{0,m+1}$ with $X_{i,j}^{0,m+1}$ for any $i$.
The claim of the lemma follows from the following relation:
 $$X_{i+1,j}^{\alpha, \beta} X_{i,j+1}^{\alpha, \beta} = X_{i,j}^{\alpha+1, \beta-1} X_{i+1,j+1}^{\alpha-1, \beta+1} + X_{i,j}^{\alpha, \beta} X_{i+1,j+1}^{\alpha, \beta}.$$
  This relation is nothing else but a Pl\"ucker relation in a ring $R_{0,\infty,0}$ where we include all vectors $v_k$, $k = - \infty, \ldots, + \infty$ and $w_k$, $k = - \infty, \ldots, + \infty$, 
  ordered so that all the $v$-s precede all the $w$-s. In other words, this is just a relation in a large enough Grassmannian, which one can identify with the universal cover of the original annulus. 
 Note also that because of the ordering on $v$-s and $w$-s, the signs in the relation are exactly as they are shown. 
\end{proof}

\begin{corollary} \label{cor:evol}
 After application of $(\mu_- \mu_+)^k$, the resulting seed $\mathfrak T_k^*$ is the one associated with the triangulation $T_k$ created by diagonals $v_i w_{i+2k}$ and $v_i w_{i+2k-1}$, $i = 1, \ldots, n$.
\end{corollary}

\begin{proof}
 Proof is by induction on $k$, the base case $k=0$ holding by definition. Applying $\mu_+$ to $\mathfrak T_k^*$ means mutating at all variables $X_{i,i+2k-1}^{\alpha, \beta}$, resulting in a seed associated with triangulation created by the diagonals 
 $v_i w_{i+2k}$ and $v_i w_{i+2k+1}$. Next, applying $\mu_-$ means mutating at all variables $X_{i,i+2k}^{\alpha, \beta}$, resulting in a seed associated with triangulation created by the diagonals $v_i w_{i+2k+2}$ and $v_i w_{i+2k+1}$,
 i.e. exactly in $\mathfrak T_{k+1}^*$.
\end{proof}

\begin{example}
 Figure~\ref{fig:zint11} shows how a single application of $\mu_- \mu_+$ looks locally.
\begin{figure}[ht]
    \begin{center}
\vspace{-.1in}
%\scalebox{1}{
\epsfig{file=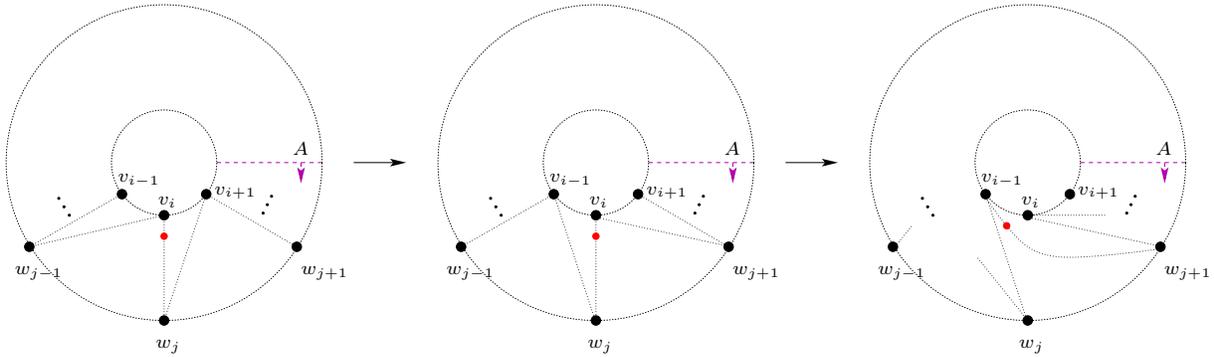, scale=1}
%\input{zint11.pstex_t} 
%}
\vspace{-.1in}
    \end{center} 
    \caption{Evolution of the narrow triangulation $T_k$ under $\mu_- \mu_+$}
    \label{fig:zint11}
\end{figure}
The red dot represents the variable $X_{i,j}^{m,1}$, which mutates into $X_{i-1,j+1}^{m,1}$.
\end{example}

The following theorem can be proved using the standard technique formulated for example in \cite[Proposition 3.6]{FP}. 
\begin{theorem} \label{thm:ca}
 The initial seed $\mathfrak T^*$ gives rise to a cluster algebra inside $R_{0,2n,1}$.
\end{theorem}
The proof requires one to check that 
\begin{itemize}
 \item all cluster variables in seeds adjacent to $\mathfrak T^*$ indeed lie in $R_{0,2n,1}$ - this has effectively  been  done in Lemma \ref{lem:locmut};
 \item all such adjacent cluster variables are relatively prime with variables in $\mathfrak T^*$.
\end{itemize}
The latter can be done similarly to how it was done in \cite{FP, FP2}. We omit the technical details.

\subsection{Integrability via Dehn twists}

Now, we can see a conceptual explanation of Zamolodchikov integrability. The key is the following easy corollary of Lemma \ref{lem:locmut} and Corollary \ref{cor:evol}.

\begin{corollary}
 We have $$(\mu_- \mu_+)^n (X_{i,j}^{\alpha, \beta}) = X_{i-n,j+n}^{\alpha, \beta}.$$
\end{corollary}

This means that the tensor diagram representing $(\mu_- \mu_+)^n (X_{i,j}^{\alpha, \beta})$ is the tensor diagram representing $X_{i,j}^{\alpha, \beta}$ to which one twice applied a Dehn twist. 
\begin{example}
 Figure~\ref{fig:zint12} gives an example for $m=4$ of what $(\mu_- \mu_+)^n$ does to a variable $X_{i-1,j}^{2,2}$. 
 \begin{figure}[ht]
    \begin{center}
\vspace{-.1in}
%\scalebox{1}{
\epsfig{file=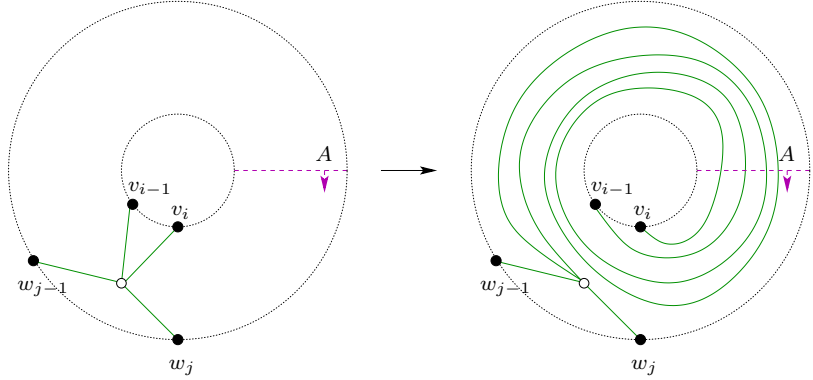, scale=1}
%\input{zint12.pstex_t} 
%}
\vspace{-.1in}
    \end{center} 
    \caption{A tensor diagram representing $X_{i-1,j}^{2,2}$ and the same diagram with Dehn twist applied twice, resulting in $X_{i-2n-1,j}^{2,2}$.}
    \label{fig:zint12}
\end{figure}
\end{example}

Now we are ready to prove Theorem \ref{thm:main}.

\begin{theorem*} 
Conjecture \ref{conj:main} holds in the case of quivers $A_m \square A_{2n-1}^{(1)}$. 
\end{theorem*}

\begin{proof}
 According to Theorem \ref{thm:aind} it is enough to prove Zamolodchikov integrability within the ring $R_{0,2n,1}$. Indeed, any collection of $2mn$ algebraically independent variables may be taken to be mutable variables of a seed $\mathfrak T^*$, 
 while setting the coefficient variables of this seed to be $1$. 
 
 Now, consider a variable $X_{i,l}^{\alpha, \beta}$ of the seed $\mathfrak T^*$, let $j = \min(\alpha, \beta)$. 
 %The parameter $j$ in the statement of Theorem \ref{thm:main} can be taken to be $\min(\alpha, \beta)$, since this is the smaller of the two distances from the end of type $A_m$ Dynkin diagram  in the box product. 
  Observe that Dehn twists $(\mu_- \mu_+)^{tn}$ insert into tensor diagram representing $X_{i,l}^{\alpha, \beta}$ a factor $\A^t \otimes \A^t \otimes \dotsc \otimes \A^t$, $j$ factors total.  Here $\A = A^{-2}$. Since $\A \in SL_{m+1}$ satisfies its own characteristic polynomial of degree 
 $m+1$, we know that the 
 vector space of all matrices $\A^{t_1} \otimes \A^{t_2} \otimes \dotsc \otimes \A^{t_j}$ is spanned by the subset of generators given by $0 \leq t_1, \ldots, t_j < m+1$. Since the number of such tensor monomials is finite, we conclude that for large enough $N$ the monomials
 $\A^t \otimes \A^t \otimes \dotsc \otimes \A^t$, $t = 0,1,\ldots, N$ are linearly dependent. 
 %This means that the $nN+1$-st column  of the Toeplitz matrix filled with  $(\mu_- \mu_+)^{t}(X_{i,l}^{\alpha, \beta})$ will be linearly dependent on the previous columns, and thus 
 %the determinant will vanish.  
\end{proof}

\begin{example}
 Consider the case $m=3$ and consider the variable $X_{i-1,j}^{2,2}$ shown in Figure~\ref{fig:zint12}. Since $\A \in SL_4$ satisfies its own characteristic polynomial, we conclude that the set of all monomials $\A^{t_1} \otimes \A^{t_2}$ is generated by its subset with 
 $0 \leq t_1, t_2 < 4$. Indeed, $\A^4$ can be expressed through smaller powers of $\A$, and one can repeatedly apply this relation to get rid of any monomial with power of $\A$ higher than $3$. Overall, we see that the dimension of the space is then at most $16$, and thus 
 if we take $17$ of the powers $\A^t \otimes \A^t$, they must be linearly dependent. 
\end{example}

The order of linearizability that follows from this proof is too large however, i.e. the sequences in question are linearizable with a smaller order than that. Theorem \ref{thm:spd} states the order of linearizability which we believe to be minimal possible. Let us give an argument proving it now.

\begin{theorem*} 
In the case of quiver $A_m \square A_{2n-1}^{(1)}$, vertex $(q,q')$  is linearizable of order $$n \cdot {m+1 \choose j}.$$  
\end{theorem*}

\begin{proof}
 The key observation is that $X_{i,l}^{\alpha, \beta}$ is an {\it {antisymmetric}} tensor in its arguments, since it is essentially the Levi-Cevita tensor. 
  Because of this, the list of monomials we used in the proof above
 $$\A^{t_1} \otimes \A^{t_2} \otimes \dotsc \otimes \A^{t_j}, \;\; 0 \leq t_1, \ldots, t_j < m+1$$
 can be shortened by requiring $t_1 < \dotsc < t_j$. The number of such monomials is ${{m+1} \choose j}$. Since it takes $n$ applications of $\mu_+ \mu_-$ to get to each next winding of the original tensor, we get order of the linear dependence to be 
 $n \cdot {{m+1} \choose j}$, as desired.

% be linearly dependent. Specifically, assume that we are plugging its $j$ legs that get Dehn twisted into a collection of vectors $u_1, \ldots, u_j$ such that some of the $u_i$-s are equal. Then we can use the facts that 
 %$$\langle \ldots, A^p u, A^q u, \ldots \rangle = - \langle \ldots, A^q u, A^p u, \ldots \rangle \text{  and  } \langle \ldots, A^p u, A^p u, \ldots \rangle = 0$$
 %to list a shorter generating set for the space of all such monomials. 
 
 %Indeed, in the places corresponding to equal $u_i$-s we want to have distinct powers of $A$, and we can assume they are ordered increasingly left to right. Assume $j = \delta n + \eta$ as in the statement of the theorem. Then as the $j$ legs ``wrap around'' the boundary component with 
 %$n$ vectors on it, exactly $\eta$ of vectors will be used $\delta +1$ times, while the remaining $n-\eta$ vectors will be used $\delta$ times. For each of the vectors used we then have ${m+1 \choose \delta +1}$ or ${m+1 \choose \delta}$ ways to assign degrees to corresponding 
 %terms in $A^{t_1} \otimes A^{t_2} \otimes \dotsc \otimes A^{t_j}$, which gives us maximum of ${m+1 \choose \delta + 1}^\eta {m+1 \choose \delta}^{n-\eta}$ for the dimension of the vector space. This means that once we have more than 
 %$n \cdot {m+1 \choose \delta + 1}^\eta {m+1 \choose \delta}^{n-\eta}$ applications of $\mu_- \mu_+$, we will see a linear dependence among the $(\mu_- \mu_+)^t (X_{i,l}^{\alpha, \beta})$, as desired.
\end{proof}

\section{Off-belt variables}

Let us now consider any other variable $X$ obtained from one of the variables in $\mathfrak T^*$ by an arbitrary sequence of mutations $\m$. Thus, $X$ lies off the ``bipartite belt'' obtained from $\mathfrak T^*$ by repeated application 
of $\mu_- \mu_+$. Nevertheless, we can still define $\mu_- \mu_+(X)$, in fact we can do it in two equivalent ways:
\begin{itemize}
 \item as a result of substitution of the variables of the seed $\mathfrak T_1^*$ (obtained from $\mathfrak T^*$ by a single step of time evolution $\mu_- \mu_+$) into the formula expressing $X$ in terms of the seed $\mathfrak T^*$;
 \item as a result of $\m \mu_- \mu_+ \m^{-1} (X)$, where $\m^{-1}$ denotes applying the sequence of mutations $\m$ in reverse order.  
\end{itemize}

\begin{theorem} \label{thm:off}
 The transformation $\mu_- \mu_+$ is linearizable at $X$ for any cluster variable $X$ of the cluster algebra with the initial seed $\mathfrak T^*$.
\end{theorem}

\begin{lemma} \label{lem:lin}
 Term-wise sum and term-wise product of two linearizable sequences are also linearizable. 
\end{lemma}

\begin{proof}
 To any linear recurrence $a_{i+n} = A_{n-1} a_{i+n-1} + \dotsc + A_1 a_{i+1} + A_0 a_i$ one can associate polynomial $P_a(t) = t^{n} - A_{n-1} t^{n-1} - \dotsc - A_0$. If we have two sequences with polynomials $P(t)$ and $Q(t)$, their sum is easily seen to satisfy 
 recurrence corresponding to the product $P(t)Q(t)$. One can also multiply polynomials in a non-standard way: if $x_1, \ldots, x_p$ are roots of $P(t)$ and $y_1, \ldots, y_q$ are roots of $Q(t)$, let $R(t)$ be the polynomial with $pq$ roots $x_i y_j$. One can express 
 the coefficients of $R(t)$ directly through the coefficients of $P(t)$ and $Q(t)$ using the Cauchy identity $$\prod_i \prod_j (1+x_i y_j) = \sum_{\lambda} s_{\lambda}(x) s_{\lambda'}(y).$$ It is easy to see that the product of two linearizable sequences with polynomials 
 $P(t)$ and $Q(t)$ is a linearizable sequence with polynomial $R(t)$.
\end{proof}

Now we are ready to prove Theorem \ref{thm:off}

\begin{proof}
 What made the proof of Zamolodchikov periodicity in the previous section work is the following fact. Each of the variables $X_{i,l}^{\alpha, \beta}$ can be written as a concatenation of three tensor diagrams: the $v$-part, the connector consisting of $j$
Kronecker tensors, and the $w$-part, see Figure~\ref{fig:zint13}. 
 \begin{figure}[ht]
    \begin{center}
\vspace{-.1in}
%\scalebox{1}{
\epsfig{file=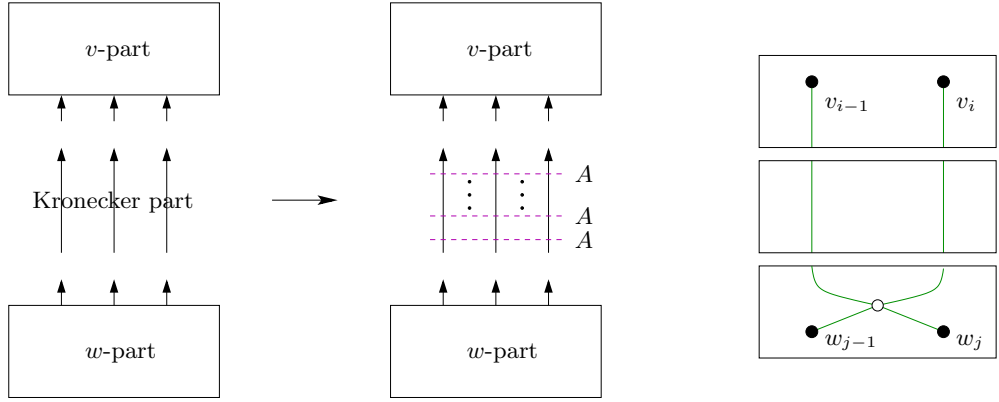, scale=1}
%\input{zint13.pstex_t} 
%}
\vspace{-.1in}
    \end{center} 
    \caption{Schematic action of $(\mu_- \mu_+)^t$ on a tensor partitioned into $w$-, Kronecker and $v$-parts; partitioning of the Levi-Cevita tensor from Figure~\ref{fig:zint12}.}
    \label{fig:zint13}
\end{figure}
Then, each application of the square of Dehn twist could be viewed as fixing the $v$- and $w$-parts, and extending the Kronecker part in the middle by $\A \otimes \A \otimes \dotsc \otimes \A$. 

It is clear that the same proof works 
for {\it {any}} invariant that can be represented in such a tensor form. It remains to be noted that according to Theorem \ref{thm:procesi} all the generators of the ring $R_{0,2n,1}$ are representable by tensors. Then so are their products, and applying Lemma \ref{lem:lin} we 
conclude that any linear combination of those products is linearizable. By Theorem \ref{thm:ca} this means that all cluster variables $X$ are linearizable. 
\end{proof}

Note that although linearizability is preserved by addition and multiplication, it is not preserved in general by division. For example, the sequence $1,2,3,\ldots$ is linearizable, while $1,1/2,1/3,\ldots$ is not. Since every variable in the cluster algebra is a {\it {rational}}
expression in the variables of seed $\mathfrak T^*$, there is no a priori reason why they should exhibit Zamolodchikov integrability. This suggests the following conjecture.

\begin{conjecture}
 Assume a recurrent quiver exhibits Zamolodchikov integrability. Then so does every element of the associated upper cluster algebra. 
\end{conjecture}

One can also treat the order of linearizability of a specific variable as a measure of complexity of this variable.  We can state the following conjecture, analogous to \cite[Conjecture 9.1]{FP} and \cite[Conjecture 21]{FP2}.
If it is true, then the order of linearizability of a cluster variable should be determined by the minimal number of strands possible in the Kronecker part of the associated tensor diagram. 

\begin{conjecture}
 All cluster variables $X$ in the cluster algebra with the initial seed $\mathfrak T^*$ can be written as (evaluations of) single tensor diagrams.
\end{conjecture}

\section{Proof of algebraic independence} \label{sec:aind}

In this section we prove Theorem \ref{thm:aind}.
Let $\tilde R_{0,2n+m+1,0}$ be the ring of invariants of $2n+m+1$ vectors: vectors $v_i$, $i=1, \ldots, n$, vectors $w_i$, $i=1, \ldots, n$ and vectors $A^iw_n$, $i = 1, \ldots, m+1$. 

\begin{lemma}
 The Krull dimension of $\tilde R_{0,2n+m+1,0}$ is $2mn+2n$.
\end{lemma}

\begin{proof}
 Starting with a generic collection of vectors $u_1, \ldots, u_{2n+m+1} \in \mathbb C^{m+1}$ one can consider the map in the reverse direction, assigning $$v_i = u_i; \;\; w_i = u_{i+n}; \;\; A = {[u_{2n+1}, \ldots, u_{2n+m+1}]}{[u_{2n}, \ldots, u_{2n+m}]}^{-1},$$
 where $[\;, \ldots, \;]$ denotes the matrix with specified columns. There is only one relation one needs to impose to get a generic set of vectors $v,w$ and a generic element $A \in SL_{m+1}$:
 $$\langle u_{2n+1}, \ldots, u_{2n+m+1} \rangle = \langle u_{2n}, \ldots, u_{2n+m}\rangle.$$ 
 Since $R_{0,2n+m+1,0}$ is the standard Pl\"ucker algebra, its dimension is well-known to be $$(m+1)((2n+m+1)-(m+1))+1 = 2mn + 2n +1.$$ Since we impose one algebraic relation, the dimension of $\tilde R_{0,2n+m+1,0}$ is one smaller than that of $R_{0,2n+m+1,0}$, as desired.  
\end{proof}

Now, in order to show that the $2mn+2n$ variables in the seed $\mathfrak T^*$ are algebraically independent, it suffices to prove that all generators of $\tilde R_{0,2n+m+1,0}$ (given by Theorem \ref{thm:procesi}) can be expressed as rational functions in elements of $\mathfrak T^*$.
All of those generators are essentially determinants and can be presented by tensor diagrams in an annulus, as described in Section \ref{sec:ann}. Therefore, the claim follows from the following stronger statement. 

\begin{theorem}
 Any tensor diagram $D$ on an annulus with $2n$ marked black vertices as above lies in the upper cluster algebra associated with $\mathfrak T^*$. In particular, it can be expressed as a Laurent polynomial in elements of $\mathfrak T^*$. 
\end{theorem}

\begin{proof}
 The proof is essentially verbatim to that of \cite[Theorem 16]{FP2}. It suffices to argue Laurentness for a seed and a collection of adjacent seeds. We shall argue it for the initial seed $\mathfrak T^*$, the argument for the adjacent seeds is similar. We need to show that by repeatedly 
 multiplying $[D]$ with elements of $\mathfrak T^*$ we can get a linear combination of monomials in elements of $\mathfrak T^*$. Let $v_i w_j$ be a diagonal connecting two marked points. The idea is to multiply $[D]$ by a sufficiently large monomial in $X_{i,j}^{\alpha, \beta}$-s. 
 %that the result becomes compatible with all of them. 
 
 Like in the proof of \cite[Theorem 16]{FP2}, we need to exhibit a local relation that allows one to get rid of crossings of $D$ with $v_i w_j$ one by one. Once this is done, all the resulting tensor diagrams are going to be contained within individual triangles of $T^*$. This can be shown to imply that they
 factor into variables of $\mathfrak T^*$ sitting on the sides of the individual triangles. 
   \begin{figure}[ht]
    \begin{center}
\vspace{-.1in}
%\scalebox{1}{
\epsfig{file=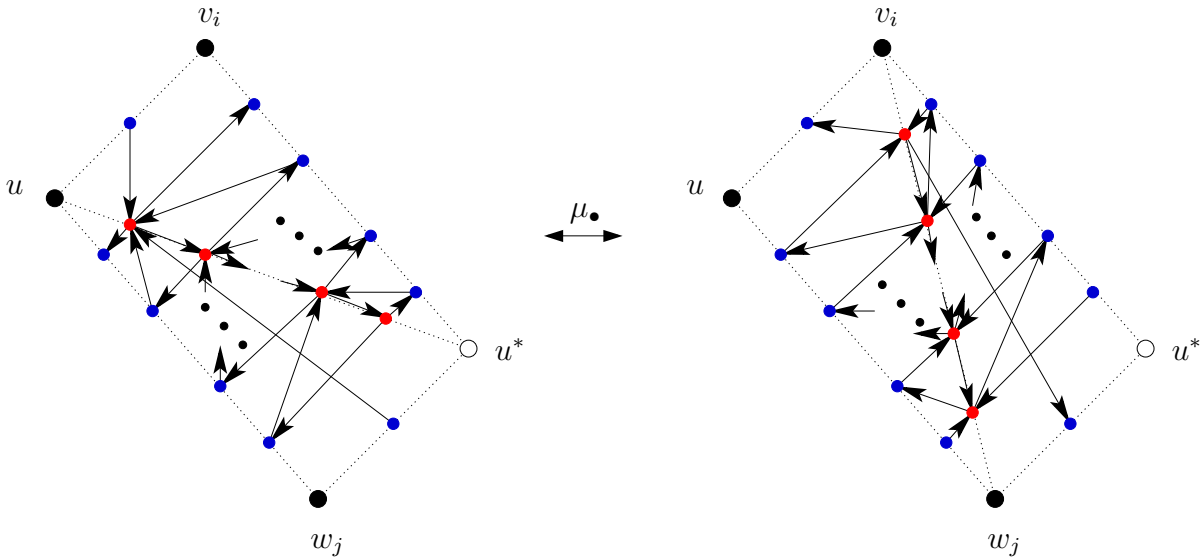, scale=1}
%\input{zint14.pstex_t} 
%}
\vspace{-.1in}
    \end{center} 
    \caption{A sequence of mutations creating the needed local relation.}
    \label{fig:zint14}
\end{figure} 
  Consider an arc $uu^*$ of $D$ crossing the diagonal $v_i w_j$. Plant the following variables on the sides of quadrilateral $u v_i u^* w_j$:
 \begin{itemize}
  \item $\langle u^*, w_j \rangle$ on the side $u^*w_j$;
  \item $\langle u, v_i, \ldots, v_{i+m} \rangle$ on side $u v_i$;
  \item $\langle u, v_i, \ldots, v_{i+\alpha-1}, w_j, \ldots,w_{j+m-\alpha-1} \rangle$, $\alpha = 0, \ldots, m-1$ on the side $u w_j$, in that order from $w_j$ to $u$;
  \item $\langle v_i, u^* \rangle$, $[[v_i, \ldots, v_{i+\alpha-1}],u^*,[w_j, \ldots, w_{j+m+1-\alpha}]]$, $\alpha = 2, \ldots, m$ on the side $v_i u^*$, in that order from $u^*$ to $v_i$.
\end{itemize}
Now, consider two triangulations of the quadrilateral $v_i u^* w_j u$. The first triangulation is with diagonal $uu^*$. Plant the following variables:
\begin{itemize}
 \item $\langle u, u^* \rangle$, $[[u, v_i, \ldots, v_{i+\alpha-2}],u^*,[w_j, \ldots, w_{j+m+1-\alpha}]]$, $\alpha = 2, \ldots, m$ on the diagonal $u u^*$, in that order from $u^*$ to $u$.
\end{itemize}
The second triangulation is with diagonal $v_i w_j$. Plant the following variables:
\begin{itemize}
 \item $\langle v_i, \ldots, v_{i+\alpha}, w_j, \ldots,w_{j+m-\alpha-1} \rangle$, $\alpha = 0, \ldots, m-1$ on the diagonal $v_i w_j$, in that order from $w_j$ to $v_i$.
\end{itemize}
Create a quiver on the created vertices as shown in Figure~\ref{fig:zint14}. We claim that the sequence $\mu_{\bullet}$ of $m$ mutations at variables on the diagonal $u u^*$ in the order from $u^*$ to $u$ changes one thus created seed into the other. 

  \begin{figure}[ht]
    \begin{center}
\vspace{-.1in}
%\scalebox{1}{
\epsfig{file=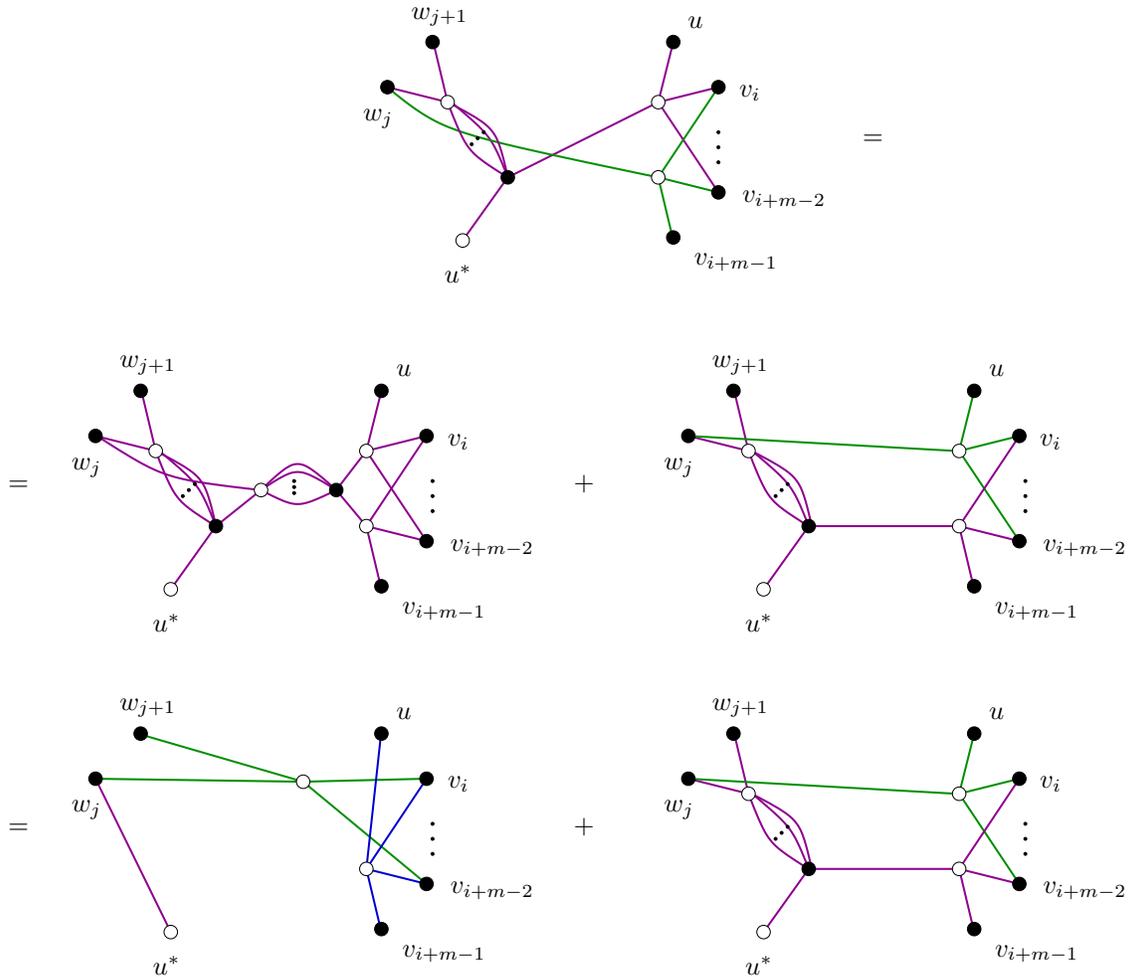, scale=1}
%\input{zint15.pstex_t} 
%}
\vspace{-.1in}
    \end{center} 
    \caption{The last mutation in the sequence $\mu_{\bullet}$ verified using tensor diagrams.}
    \label{fig:zint15}
\end{figure} 
 This is easily verified using the skein relation for tensor diagrams. For example, the last mutation in the sequence $\mu_{\bullet}$ has form 
 $$[[u, v_i, \ldots, v_{i+m-2}],u^*,[w_j, w_{j+1}]] \; \langle v_i, \ldots, v_{i+m-1}, w_j \rangle =$$ $$= \langle u, w_j \rangle \; \langle v_i, \ldots, v_{i+m-2}, w_j, w_{j+1} \rangle \; \langle v_i, \ldots, v_{i+m-1} \rangle \; + $$    
 $$+ [[v_i, \ldots, v_{i+m-1}],u^*,[w_j, w_{j+1}]]  \; \langle u, v_i, \ldots, v_{i+m-2}, w_j \rangle$$
 and is shown in Figure~\ref{fig:zint15}. Here each tensor diagram $D$ denotes the associated normalized invariant $[[D]]$. 
 
 Note that for odd $m$ the formulas hold only with the correct choice of sign for each tensor diagram. However, since the goal of the argument is to 
 show that there exists a Laurent expression for the variable $\langle u^*, u \rangle$ in terms of the variables of the other seed, the exact signs do not matter. This goal is achieved, as we can reverse $\mu_{\bullet}$ and thus obtain the 
 needed Laurent expression for $\langle u^*, u \rangle$. This completes the argument, as all the variables in the second triangulation represent tensor diagrams that do not cross $v_i w_j$.  
\end{proof}
 
%\begin{remark}
% We determined the Krull dimension of the ring $R_{0,2n,1}$. It seems likely that similar technique would allow to determine Krull dimension of more general rings $R_{a,b,c}$. 
%\end{remark}

%\begin{conjecture}
% For $c \geq 1$ the Krull dimension of $R_{a,b,c}$ is $(a+b)(m+1)+3m(c-1)$. 
%\end{conjecture}

\end{document}